\documentclass[10pt,reqno]{article}
\usepackage{latexsym, amsmath, amssymb, a4, epsfig}
\usepackage{graphicx}

\newtheorem{thm}{Theorem}[section]

\newtheorem{lem}[thm]{Lemma}

\newcommand{\la}{\langle}
\newcommand{\ra}{\rangle}
\newcommand{\ie}{{\it i.e.}}
\newcommand{\Om}{\Omega}
\newcommand{\ep}{\epsilon}
\newcommand{\RR}{\mathbb{R}}
\newcommand{\ZZ}{\mathbb{Z}}
\newcommand{\CC}{\mathbb{C}}
\newcommand{\vp}{\varphi}

\newcommand{\p}{\partial}

\newcommand{\Scal}{\mathcal{S}}
\newcommand{\Kcal}{\mathcal{K}}
\newcommand{\Dcal}{\mathcal{D}}
\newcommand{\de}{\delta}
\newcommand{\lam}{\lambda}

\newcommand{\KK}{\mathbb{K}}
\newcommand{\II}{\mathbb{I}}
\newcommand{\SSS}{\mathbb{S}}
\newcommand{\AAA}{\mathbb{A}}
\newcommand{\Hcal}{\mathcal{H}}

\newcommand{\eqnref}[1]{(\ref {#1})}
\newcommand{\ds}{\displaystyle}
\newcommand{\nm}{\noalign{\smallskip}}
\newcommand{\pd}[2]{\frac {\p #1}{\p #2}}
\newcommand{\beq}{\begin{equation}}
\newcommand{\eeq}{\end{equation}}

\newcommand{\Ker}{{\rm Ker}}
\newcommand{\Range}{{\rm Range}}
\newcommand{\diSe}{\dfrac{\p}{\p\nu_i}\Scal_{\Gamma_e}}
\newcommand{\deSi}{\dfrac{\p}{\p\nu_e}\Scal_{\Gamma_i}}

\newenvironment{proof}{
  \noindent{\it Proof.}\ }{\hspace*{\fill}
  \begin{math}\Box\end{math}\medskip}
\newcommand{\qed}{\hfill $\Box$ \medskip}

\begin{document}
\title{Spectral theory of a Neumann-Poincar\'e-type operator and analysis of
cloaking due to anomalous localized resonance\thanks{\footnotesize This work was
supported by the ERC Advanced Grant Project MULTIMOD--267184 and
NRF grants No. 2009-0090250, 2010-0004091, and 2010-0017532, and by the NSF
through grant DMS-0707978}}

\author{Habib Ammari\thanks{\footnotesize Department of Mathematics and Applications, Ecole Normale Sup\'erieure,
45 Rue d'Ulm, 75005 Paris, France (habib.ammari@ens.fr).} \and
Giulio Ciraolo\thanks{\footnotesize  Dipartimento di Matematica e
Informatica, Universit\`a di Palermo Via Archirafi 34, 90123, Palermo, Italy
  (g.ciraolo@math.unipa.it).} \and Hyeonbae
Kang\thanks{Department of Mathematics, Inha University, Incheon
402-751, Korea (hbkang@inha.ac.kr, hdlee@inha.ac.kr).}  \and
Hyundae Lee\footnotemark[4]  \and Graeme W.
Milton\thanks{\footnotesize Department of Mathematics, University
of Utah, Salt Lake City, UT 84112, USA (milton@math.utah.edu).}}

\maketitle

\begin{abstract}
The aim of this paper is to give a mathematical justification of
cloaking due to anomalous localized resonance (CALR).  We consider the
dielectric problem with a source term in a  structure with a
layer of plasmonic material. Using layer potentials and symmetrization
techniques, we give a necessary and sufficient condition on the fixed
source term for electromagnetic power dissipation to blow up as the
loss parameter of the plasmonic material goes to zero. This condition is written
in terms of  the Newtonian potential of the source term. In the case
of concentric disks, we make the condition even more
explicit. Using the condition, we are able to show that for any
source supported outside a critical radius CALR
does not take place, and for sources located
inside the critical radius satisfying certain conditions CALR does take place
as the loss parameter goes to zero.
\end{abstract}

\noindent {\footnotesize {\bf AMS subject classifications.} 35R30,
35B30}

\noindent {\footnotesize {\bf Key words.} anomalous localized
resonance, plasmonic materials, singular perturbation, non-self-adjoint
operator, symmetrization, quasi-static cloaking}

\section{Introduction}

In recent years much interest has been aroused by the possibility
of cloaking objects from interrogation by electromagnetic waves.
Many schemes are under active current investigation
\cite{glu,alu,leonhardt,pendry,miller,CCKSM_07,kohn1,GKLU,broaband,liu,GMO_09,LCZC_09,kohn2}. One
such scheme, which is the focus of our study, relies on resonant interaction to mask the
electromagnetic signature of the object to be cloaked
\cite{MN_PRSA_06,osa,bruno,MNM_07,NMET_08,MNMCJ_09,MNBM_09,bouchitte,NMB_11}. 

We consider the
dielectric problem with a source term $\alpha f$, proportional to $f$, which models the
quasi-static (zero-frequency) transverse magnetic regime. The
cloaking of the source is achieved in a region external to a plasmonic
structure. The plasmonic structure consists of a shell having
relative permittivity $-1+i \delta$ with $\delta$ modelling
losses.

The cloaking issue is directly linked to the existence of
anomalous localized resonance (ALR), which is tied to the fact
that an elliptic system of equations can exhibit localization
effects near the boundary of ellipticity. The plasmonic structure
exhibits ALR if, as the loss parameter $\delta$ goes to zero, the
magnitude of the quasi-static in-plane electric field diverges
throughout a specific region (with sharp boundary not defined by
any discontinuities in the relative permittivity), called the
anomalous resonance region, but converges to a smooth field
outside that region. The convergence to a smooth field outside
the region was shown in \cite{NMM_94}, where the first numerical evidence
for ALR was also presented. A proof of ALR for a dipolar source outside
a plasmonic annulus was given
in \cite{MNMP_PRSA_05}.

 Alexei Efros (2005 private communication to GWM)
made the key observation that for a fixed dipolar source within a critical distance
of the plasmonic structure  the total electrical power
absorbed would become infinite as $\delta\to 0$, which is unphysical. The anomalously
resonant fields interact with the source creating a sort of ``electromagnetic molasses'' against which the
source has to a huge amount of work to maintain its amplitude, in fact an infinite amount
of work in the limit $\delta\to 0$. Therefore
it makes sense to normalize the source term (by adjusting $\alpha$, letting it depend on $\delta$)
so the source supplies power at constant rate
independent of $\delta$. Then outside the region where ALR occurs the field tends to
zero as $\delta\to 0$: the source becomes cloaked. Cloaking also extends
to finite collections of polarizable dipoles (dipole sources whose strength is
proportional the field acting on them) within a critical radius
around a plasmonic annulus \cite{MN_PRSA_06,osa}, and to a sufficiently small
dielectric disk (with radius which goes to zero as $\delta\to 0$) lying
within this critical radius \cite{bouchitte}. However numerical evidence
suggests that a small dielectric disk with $\delta$ independent radius
is only partially cloaked in the limit $\delta\to 0$ \cite{bruno}. We also mention that
opposing sources on opposite sides of a planar superlens can be cloaked \cite{board} but this is
due to cancellation of fields, rather than anomalous resonance. 

To mathematically state the problem, let $\Om$ be a bounded domain
in $\RR^2$ and let $D$ be a domain whose closure is contained in
$\Om$. Throughout this paper, we assume that $\Omega$ and $D$ are
of class $ \mathcal{C}^{1,\mu}$ for some $0<\mu <1$. For a given loss parameter $\delta>0$, the permittivity
distribution in $\RR^2$ is given by
 \beq
 \ep_\delta  = \begin{cases}
 1 \quad & \mbox{in } \RR^2 \setminus \overline{\Om}, \\
 -1+ i \delta \quad & \mbox{in } \Om \setminus \overline{D}, \\
 1 \quad &\mbox{in } D.
 \end{cases}
 \eeq
We may consider the configuration as a core with permittivity 1 coated by the shell $\Om \setminus \overline{D}$ with permittivity $-1+i\delta$. For a given
function $f$ compactly supported in $\RR^2$ satisfying
 \beq\label{zeroint}
 \int_{\RR^2} f dx=0
 \eeq
(which physically is required by conservation of charge), we consider the following dielectric problem:
 \beq \label{basiceqn}
 \nabla \cdot \ep_\delta \nabla V_\delta =\alpha f \quad \mbox{in } \RR^2,
 \eeq
with the decay condition $V_\delta (x) \to 0$ as $|x| \to \infty$.

A fundamental problem is to identify those sources $f$ such that when $\alpha=1$
then first
 \beq\label{blowup1}
E_\delta := \int_{\Om \setminus \overline{D}} \delta |\nabla V_\delta|^2 dx \to \infty \quad\mbox{as } \delta \to 0.
 \eeq
and second $V_\delta$ remains bounded outside some radius $a$:
\beq\label{boundedness}
|V_\delta(x)|<C,~~~{\rm when}~~|x|>a
\eeq
for some constants $C$ and $a$ independent of $\delta$ (which necessitates
that the ball $B_a$ contains the entire region of  anomalous localized resonance).
The quantity $E_\delta$ is proportional to the electromagnetic power dissipated into
heat by the time harmonic electrical field
averaged over time. Hence \eqnref{blowup1} implies an infinite amount of energy dissipated
per unit time in the limit $\delta \to 0$ which is unphysical. If instead we choose $\alpha=1/\sqrt{E_\delta}$ then
the source $\alpha f$ will produce the same power independent of $\delta$
and the new associated solution $V_\delta$ (which is the previous solution  $V_\delta$ multiplied by $\alpha$)
will approach zero outside the radius $a$:
cloaking due to anomalous localized resonance (CALR) occurs. The conditions \eqnref{blowup1} and \eqnref{boundedness}
are sufficient to ensure CALR: a necessary and sufficient condition is that (with $\alpha=1$)
$V_\delta/\sqrt{E_\delta}$ goes to zero outside some radius as  $\delta \to 0$. We also consider a weaker blow-up of the energy dissipation, namely,
 \beq\label{weakerblowup}
 \limsup_{\delta \to 0} E_\delta = \infty.
 \eeq
We say that weak CALR takes place if \eqnref{weakerblowup} holds (in addition to \eqnref{boundedness}). Then the (renormalized) source $f/\sqrt{E_\delta}$ will be essentially invisible
at a infinite sequence of small values of $\delta$ tending to zero (but would be quite visible for values of $\delta$ interspersed between this sequence if CALR does not additionaly hold).

The aim of this paper is to develop a general method based on the potential theory to study cloaking due to anomalous resonance. Using layer potential techniques, we reduce the problem to a
singularly perturbed system of integral equations. The system is
non-self-adjoint. A symmetrization technique is introduced in order
to express the solution in terms of the eigenfunctions of a
self-adjoint compact operator.  The symmetrization technique is
based on a generalization of a Calder\'on identity to the system of integral equations
under consideration and a general theorem on symmetrization of non-selfadjoint operators obtained in a
 recent paper by Khavinson {\it et al} \cite{KPS}.

Using the technique developed in this paper, we are able to
provide a necessary and sufficient condition on the source term
under which the blowup \eqnref{blowup1} of the power dissipation
takes place. The condition is given in terms
of the Newtonian potential of the source, which is the solution for
the potential in the absence of the plasmonic structure.

In the case of an annulus ($D$ is the disk of radius
$r_i$ and $\Om$ is the concentric disk of radius $r_e$), it is known \cite{MN_PRSA_06}
that there exists a critical radius (the cloaking radius)
 \beq\label{ALRR}
 r_* = \sqrt{{r_e^3}{r_i}^{-1}}.
 \eeq
such that any finite collection of dipole sources located at fixed positions within
the annulus $B_{r_*}\setminus \overline{B}_e$ is cloaked. We show (see Theorem \ref{corollary1} below)
that if $f$ is an integrable function supported in $E \subset B_{r_*}\setminus \overline{B}_e$ satisfying \eqnref{zeroint} and the Newtonian potential of $f$ does not extend as a harmonic function in $B_{r_*}$, then weak CALR takes place. Moreover, we show that if the Fourier coefficients of the Newtonian potential of $f$ satisfy a mild gap condition,  then CALR takes place. Using this result, we are able to show that a quadrupole source inside the annulus $B_{r_*}\setminus \overline{B}_e$
would be cloaked, in agreement with the numerical results of \cite{osa}. Conversely we show that if the source function $f$ is supported
outside $B_{r_*}$ then \eqnref{blowup1} does not happen and no cloaking occurs.
 We stress that we assume $f$ does not depend on $\delta$:
the results of \cite{bruno} strongly suggest that there exist sequences of sources $f_\delta$ supported in $E \subset B_{r_*}\setminus \overline{B}_e$
with non-trivial Newtonian potentials outside $E$, such that the power dissipation does not blow up, and such that $V_\delta$ does
not go to zero outside $B_{r_*}$ as $\delta\to 0$.

This paper is organized as follows. In Section 2 we transform the
problem into a system of integral equations using layer
potentials. In Section 3, we develop a spectral theory for the
relevant integral operators and derive a necessary and sufficient
condition for CALR to take place. Section 4 treats the special
case of an annulus.

%
\section{Layer potential formulation}
%

Let $G$ be the fundamental solution to the Laplacian in two dimensions which is given by
$$ G(x) = \ds \frac{1}{2\pi} \ln |x|. $$
Let $\Gamma_i:= \p D$ and $\Gamma_e:= \p\Om$. For $\Gamma=\Gamma_i$ or $\Gamma_e$, we denote, respectively, the single
and double layer potentials of a function $\vp \in L^2(\Gamma)$
as $\Scal_\Gamma[\vp]$ and $\Dcal_\Gamma[\vp]$, where
\begin{align*}
\Scal_\Gamma[\vp] (x) &:= \int_{\Gamma} G(x-y) \vp (y) \, d
\sigma(y), \quad x \in \RR^2, \\
\nm \Dcal_\Gamma[\vp] (x) &: =\int_{\Gamma} \pd{}{\nu (y)}G (x-y)
\vp (y) \, d \sigma(y) \;, \quad x \in \RR^2 \setminus \Gamma.
\end{align*} Here, $\nu(y)$ is the outward unit normal  to $\Gamma$ at $y$.

We also define a boundary integral operator $\Kcal_\Gamma$ on $L^2(\Gamma)$ by
 $$
 \Kcal_\Gamma[\vp] (x) := \frac{1}{2\pi} \int_{\Gamma} \frac{\la
 y -x, \nu(y) \ra}{|x-y|^2} \vp (y)\,d\sigma(y),
 $$
and let
$\Kcal_\Gamma^*$ be the  $L^2$-adjoint of $\Kcal_\Gamma$. Hence,
the operator $\Kcal_\Gamma^*$ is given by $$ \Kcal_\Gamma^*[\vp]
(x) = \frac{1}{2\pi} \int_{\Gamma} \frac{\la
 x -y, \nu(x) \ra}{|x-y|^2} \vp (y)\,d\sigma(y), \quad \vp \in  L^2(\Gamma).
 $$
Here and throughout this paper, $\la \;, \;\ra$ denotes the scalar
product in $\RR^2$. The operators $\Kcal_\Gamma$ and
$\Kcal_\Gamma^*$ are sometimes called Neumann-Poincar\'e
operators. These operators are compact in $L^2(\Gamma)$ if
$\Gamma$ is $\mathcal{C}^{1, \alpha}$ for some $\alpha>0$.

The following notation will be used throughout this paper. For a
function $u$ defined on $\RR^2 \setminus \Gamma$, we denote
 $$
 u |_\pm(x) := \lim_{t \to 0^+}  u(x \pm t \nu(x)), \quad x \in \Gamma ,
 $$
and
 $$
 \pd{u}{\nu} \Big |_\pm(x) := \lim_{t \to 0^+} \la \nabla u(x \pm t \nu(x)), \nu(x) \ra\;, \quad x \in \Gamma,
 $$
if the limits exist.

The following jump formulas relate the traces of the double layer
potential and the normal derivative of the single layer potential
to the operators $\Kcal_\Gamma$ and $\Kcal_\Gamma^*$. We have
 \begin{align}
 (\Dcal_\Gamma [\vp])  |_{\pm} (x) & = \left( \mp \frac{1}{2} I
 + \Kcal_\Gamma \right) [\vp] (x), \quad  x \in \Gamma, \label{doublejump} \\
 \pd{}{\nu} \Scal_\Gamma [\vp] \Big |_\pm (x) & = \left( \pm
 \frac{1}{2} I + \Kcal_\Gamma^* \right) [\vp] (x), \quad
 x \in \Gamma. \label{singlejump}
 \end{align}
See, for example, \cite{book2, Folland76}.

Let $F$ be the Newtonian potential of $f$, {\it i.e.},
 \beq \label{newton}
 F(x)= \int_{\RR^2} G(x-y) f(y) dy, \quad x \in \RR^2.
 \eeq
Then $F$ satisfies $\Delta F=f$ in $\RR^2$, and the solution $V_\delta$ to \eqnref{basiceqn} may be
represented as
 \beq\label{vdelta}
 V_\delta(x) = F(x) + \Scal_{\Gamma_i}[\vp_i](x) + \Scal_{\Gamma_e}[\vp_e](x)
 \eeq
for some functions $\vp_i \in L^2_0(\Gamma_i)$ and $\vp_e \in
L^2_0(\Gamma_e)$ ($L^2_0$ is the collection of all square integrable functions with the integral zero). The transmission conditions along the interfaces $\Gamma_e$ and $\Gamma_i$ satisfied by
$V_\delta$ read
 \begin{align*}
 (-1+i\delta) \pd{V_\delta}{\nu} \Big|_{+} =  \pd{V_\delta}{\nu}
 \Big|_{-} \quad \mbox{on } \Gamma_i \\
 \pd{V_\delta}{\nu} \Big|_{+} = (-1+i\delta) \pd{V_\delta}{\nu} \Big|_{-} \quad \mbox{on }
 \Gamma_e.
 \end{align*}
Hence the pair of potentials $(\vp_i, \vp_e)$  is the solution to
the following system of integral equations:
 $$
 \begin{cases}
 \ds (-1+i\delta) \pd{\Scal_{\Gamma_i}[\vp_i]}{\nu_i} \Big|_{+} -
  \pd{\Scal_{\Gamma_i}[\vp_i]}{\nu_i} \Big|_{-} +
 (-2+i\delta) \pd{\Scal_{\Gamma_e}[\vp_e]}{\nu_i}  = (2-i\delta) \pd{F}{\nu_i} \quad \mbox{on } \Gamma_i, \\
 \nm
 \ds (2-i\delta) \pd{\Scal_{\Gamma_i}[\vp_i]}{\nu_e} +
 \pd{\Scal_{\Gamma_e}[\vp_e]}{\nu_e} \Big|_{+} -
 (-1+i\delta) \pd{\Scal_{\Gamma_e}[\vp_e]}{\nu_e} \Big|_{-}
 = (-2+i\delta) \pd{F}{\nu_e} \quad \mbox{on } \Gamma_e.
 \end{cases}
 $$
Note that we have used the notation $\nu_i$ and $\nu_e$ to indicate
the outward normal on $\Gamma_i$ and $\Gamma_e$, respectively.
Using the jump formula \eqnref{singlejump} for the normal derivative of the single
layer potentials, the above equations can be rewritten as
 \beq \label{matrixeqn}
 \begin{bmatrix}
 \ds -z_\delta I +
 \Kcal_{\Gamma_i}^* & \ds \pd{}{\nu_i} \Scal_{\Gamma_e} \\
 \nm \ds \pd{}{\nu_e} \Scal_{\Gamma_i} & \ds z_\delta I +
 \Kcal_{\Gamma_e}^* \end{bmatrix}
 \begin{bmatrix}
 \vp_i \\ \vp_e
 \end{bmatrix}  = - \begin{bmatrix}
 \ds \pd{F}{\nu_i} \\ \nm \ds \pd{F}{\nu_e}
 \end{bmatrix}
 \eeq
on $L^2_0(\Gamma_i) \times L^2_0(\Gamma_e)$, where we set
\begin{equation}\label{lam e mu}
z_\delta  = \frac{i\de}{2(2-i\de)}.
\end{equation}
Note that the operator in
\eqnref{matrixeqn} can be viewed as a compact perturbation of the
operator
 \beq \label{mainoperator}
 R_\delta:= \begin{bmatrix}
 \ds -z_\delta I +
 \Kcal_{\Gamma_i}^* & 0 \\
 \nm 0 & \ds z_\delta I +
 \Kcal_{\Gamma_e}^* \end{bmatrix} \, .
 \eeq

We now recall Kellogg's result in \cite{kellog} on the spectrums
of $\Kcal_{\Gamma_i}^*$ and $\Kcal_{\Gamma_e}^*$. The eigenvalues
of $\Kcal_{\Gamma_i}^*$ and $\Kcal_{\Gamma_e}^*$
lie in the interval $]-\frac{1}{2}, \frac{1}{2}]$. Observe
that $z_\delta \to 0$
as $\delta \to 0$ and that there are sequences of eigenvalues of $\Kcal_{\Gamma_i}^*$ and
$\Kcal_{\Gamma_e}^*$ approaching $0$ since $\Kcal_{\Gamma_i}^*$ and $\Kcal_{\Gamma_e}^*$ are compact. So $0$ is the essential
singularity of the operator valued meromorphic function
 $$
 \lambda \in \CC \mapsto (\lambda I + \Kcal_{\Gamma_e}^*)^{-1}.
 $$
This causes a serious difficulty in dealing with \eqnref{matrixeqn}. We emphasize that
$\Kcal_{\Gamma_e}^*$ is not self-adjoint in general.
In fact, $\Kcal_{\Gamma_e}^*$ is self-adjoint only when $\Gamma_e$
is a circle or a sphere (see \cite{Lim_IJM_01}).

Let $\Hcal = L^2(\Gamma_i) \times L^2(\Gamma_e)$. We write
\eqref{matrixeqn} in a slightly different form. We first apply the
operator
\begin{equation*}
\begin{bmatrix} -I & 0 \\ 0 & I \end{bmatrix} :\ \Hcal \to \Hcal
\end{equation*}
to \eqref{matrixeqn}. Then the equation becomes
\begin{equation}\label{matrixeq3}
\begin{bmatrix}
 \ds z_\delta  I - \Kcal_{\Gamma_i}^* & & \ds - \pd{}{\nu_i} \Scal_{\Gamma_e} \\
 \nm \ds \pd{}{\nu_e} \Scal_{\Gamma_i} & & z_\delta  I + \ds \Kcal_{\Gamma_e}^*
 \end{bmatrix}  \begin{bmatrix} \ds \vp_i \\ \vp_e \end{bmatrix} = \begin{bmatrix} \ds \frac{\p F}{\p \nu_i} \\ \nm \ds - \frac{\p F}{\p \nu_e} \end{bmatrix}.
\end{equation}
Let the Neumann-Poincar\'e-type operator $\KK^* : \Hcal \to \Hcal$ be
defined by
\begin{equation} \label{eq:K*}
 \KK^*:= \begin{bmatrix}
 \ds -\Kcal_{\Gamma_i}^* & \ds -\pd{}{\nu_i} \Scal_{\Gamma_e} \\
 \nm \ds \pd{}{\nu_e} \Scal_{\Gamma_i} & \ds \Kcal_{\Gamma_e}^*
 \end{bmatrix},
\end{equation}
and let
\begin{equation} \label{defg}
\Phi:=\begin{bmatrix} \ds \vp_i \\ \vp_e \end{bmatrix} , \quad
g:= \begin{bmatrix} \ds \frac{\p F}{\p \nu_i} \\ \nm \ds -
\frac{\p F}{\p \nu_e} \end{bmatrix}.
\end{equation}
Then, \eqref{matrixeq3} can be rewritten in the form
\begin{equation}\label{matrixeq4 ep=1}
(z_\delta  \II + \KK^* ) \Phi = g,
\end{equation}
where $\II$ is given by
\begin{equation}\label{II}
\II = \begin{bmatrix}
 \ds  I  & 0 \\
 0  &  I
 \end{bmatrix} .
\end{equation}

\section{Properties of $\KK^*$}

In the following we provide some properties of $\KK^*$. In
particular, we compute the adjoint operator $\KK$ of $\KK^*$,
study the spectrum of $\KK^*$, and show that $\KK^*$ is
symmetrizable on the space $\Hcal = L^2(\Gamma_i) \times L^2(\Gamma_e)$.

\subsection{Adjoint operator of $\KK^*$}
We first compute the adjoint of $\KK^*$. Denote by $\langle ,
\rangle_{L^2(\Gamma)} $ the Hermitian product on $L^2(\Gamma)$ for
$\Gamma=\Gamma_i$ or $\Gamma_e$. It is easy to see that
\begin{equation} \label{adjoint diSe}
\langle \pd{}{\nu_i} \Scal_{\Gamma_e}[\vp_e], \psi_i
\rangle_{L^2(\Gamma_i)} = \langle \vp_e, \Dcal_{\Gamma_i}[\psi_i]
\rangle_{L^2(\Gamma_e)},
\end{equation}
and
\begin{equation} \label{adjoint deSi}
\langle \pd{}{\nu_e} \Scal_{\Gamma_i} [\vp_i], \psi_e
\rangle_{L^2(\Gamma_e)} = \langle \vp_i, \Dcal_{\Gamma_e}[\psi_e]
\rangle_{L^2(\Gamma_i)}.
\end{equation}
Thus the $L^2$-adjoint of $\KK^*$, $\KK$, is given by
\begin{equation} \label{eq:A*}
 \KK = \begin{bmatrix}
 \ds - \Kcal_{\Gamma_i} & \ds \Dcal_{\Gamma_e} \\
 \nm \ds - \Dcal_{\Gamma_i} & \ds \Kcal_{\Gamma_e}
 \end{bmatrix}.
\end{equation}
We emphasize that the operators $\Dcal_{\Gamma_e}$ and $\Dcal_{\Gamma_i}$ in the off-diagonal entries are those from $L^2(\Gamma_e)$ into $L^2(\Gamma_i)$, and from $L^2(\Gamma_i)$ into $L^2(\Gamma_e)$, respectively.

\subsection{Spectrum of $\KK^*$}

We now look into the spectrum of $\KK^*$. We have the following
proposition which is a generalization of Kellogg's result in
\cite{kellog} on the spectrum of the operator $\Kcal^*_{\Gamma}$ on $L^2(\Gamma)$.

\begin{lem} \label{prop eigenvalues JA}
The spectrum of $\KK^*$ lies in the interval $[-1/2,1/2]$.
\end{lem}

\begin{proof}
Let $\lam$ be a point in
the spectrum of $\KK^*$. Then there exists
$\Phi=(\vp_i,\vp_e)$ with $\vp_i\in L^2(\Gamma_i)$ and $\vp_e
\in L^2(\Gamma_e)$ such that
\begin{equation*}
\begin{cases}
\ds \Kcal_{\Gamma_i}^*[\vp_i] + \pd{}{\nu_i} \Scal_{\Gamma_e} [\vp_e] = -\lam \vp_i \quad &\mbox{on }\Gamma_i,\\
\nm
\ds \pd{}{\nu_e} \Scal_{\Gamma_i}[\vp_i] + \Kcal_{\Gamma_e}^*
[\vp_e] = \lam \vp_e  & \mbox{on }\Gamma_e.
\end{cases}
\end{equation*}
By integrating the above equations  on $\Gamma_i$ and $\Gamma_e$,
respectively, and using \eqref{adjoint diSe} and \eqref{adjoint
deSi}, we obtain
\begin{equation*}
\begin{cases}
 \ds \big(\lam + \frac{1}{2}\big) \int_{\Gamma_i} \vp_i d\sigma = 0,& \\
 \nm
 \ds \big(\lam - \frac{1}{2}\big) \int_{\Gamma_e} \vp_e d\sigma
= - \int_{\Gamma_i} \vp_i d\sigma. &
\end{cases}
\end{equation*}
Here, we used the facts that
$\Kcal_{\Gamma_i}[1]= 1/2$, $\Kcal_{\Gamma_e}[1]=1/2$,
$\Dcal_{\Gamma_e}[1]=1$ on $\Gamma_i$, and $\Dcal_{\Gamma_i}[1]=0$
on $\Gamma_e$. Thus, either $\lam=\pm 1/2$ or $\lam \neq \pm 1/2$
with $\vp_i \in L_0^2(\Gamma_i)$ and $\vp_e \in
L_0^2(\Gamma_e)$ holds. We assume that $\lam \neq \pm 1/2$ and
consider
\begin{equation*}
u(x): = \Scal_{\Gamma_i}[\vp_i](x) + \Scal_{\Gamma_e}[\vp_e](x), \quad x \in \RR^2.
\end{equation*}
Since $\vp_i \in L_0^2(\Gamma_i)$ and $\vp_e \in
L_0^2(\Gamma_e)$, we have $u(x)=O(|x|^{-1})$ as $|x| \to \infty$, and hence the following integrals are finite:
\begin{equation*}
A=\int_D |\nabla u|^2 dx,\quad B=\int_{\Omega\setminus
\overline{D}} |\nabla u|^2 dx,\quad C=\int_{\RR^2 \setminus
\overline{\Omega}} |\nabla u|^2 dx.
\end{equation*}
Since $\lam$ is an eigenvalue of $\KK^*$, we obtain from Green's formulas and the jump relation \eqnref{singlejump} that
\begin{equation*}
A=-\big(\lam +\frac{1}{2} \big) \int_{\Gamma_i} \bar{u} \vp_i d\sigma,
\end{equation*}
\begin{equation*}
B= \big(\lam - \frac{1}{2} \big) \int_{\Gamma_i} \bar{u} \vp_i d\sigma
+ \big(\lam - \frac{1}{2} \big) \int_{\Gamma_e} \bar{u} \vp_e d\sigma
,
\end{equation*}
and
\begin{equation*}
C= -\big(\lam + \frac{1}{2} \big) \int_{\Gamma_e} \bar{u} \vp_e d\sigma .
\end{equation*}
Thus, we get
\begin{equation*}
\frac{\lam - \frac{1}{2}}{\lam + \frac{1}{2}} (A+C) = -B,
\end{equation*}
which implies
\begin{equation*}
\lam = \frac{1}{2} - \frac{B}{A+B+C}.
\end{equation*}
Since $A,B,C \geq 0$ and $A+B+C>0$, we conclude that $-1/2 < \lambda < 1/2$. This completes the proof.
\end{proof}

\subsection{Calder\'on's identity}

We prove that there exists a positive self-adjoint
operator $-\SSS$ such that $\SSS\KK^*=\KK\SSS$ on $\Hcal =
L^2(\Gamma_i) \times L^2(\Gamma_e)$. This is a Calder\'on-type
identity. It will be used to prove that $\KK^*$ is symmetrizable.

In fact, $\SSS$ is given by
\begin{equation} \label{eq S}
\SSS = \begin{bmatrix}
\Scal_{\Gamma_i} & \Scal_{\Gamma_e} \\
\Scal_{\Gamma_i} & \Scal_{\Gamma_e}
\end{bmatrix}.
\end{equation}
Again we emphasize that the operator $\Scal_{\Gamma_e}$ off the diagonal is the
one from $L^2(\Gamma_e)$ into $L^2(\Gamma_i)$, and likewise for $\Scal_{\Gamma_i}$ off the diagonal.

\begin{lem} \label{lemma S positive}
The operator $-\SSS$ is self-adjoint and $-\SSS
\geq 0$ on $\Hcal$.
\end{lem}

\begin{proof}
It is clear that $\ds \begin{bmatrix}
\Scal_{\Gamma_i} & 0 \\
0 & \Scal_{\Gamma_e}
\end{bmatrix}$
is self-adjoint. On the other hand, from the relations
\begin{equation*}
\langle \Scal_{\Gamma_i} [\vp_i], \vp_e \rangle_{L^2(\Gamma_e)}
= \langle \vp_i, \Scal_{\Gamma_e} [\vp_e] \rangle_{L^2(\Gamma_i)}
\end{equation*}
and
\begin{equation*}
\langle \Scal_{\Gamma_e} [\vp_e], \vp_i \rangle_{L^2(\Gamma_i)}
= \langle \vp_e, \Scal_{\Gamma_i} [\vp_i] \rangle_{L^2(\Gamma_e)},
\end{equation*}
it follows that
$\begin{bmatrix}
0 & \Scal_{\Gamma_e} \\
\Scal_{\Gamma_i} & 0
\end{bmatrix}$
is self-adjoint and hence $\SSS$ is self-adjoint.

Let $\Phi=(\vp_i,\vp_e)\in \Hcal$ and define
\begin{equation} \label{u proof S>0}
u(x) = \Scal_{\Gamma_i} [\vp_i](x) + \Scal_{\Gamma_e} [\vp_e](x).
\end{equation}
Then we have
\begin{equation*}
\int_D|\nabla u|^2 dx = \int_{\p D} \bar{u}  \Big(
-\frac{1}{2}\vp_i + \Kcal_{\Gamma_i}^* [\vp_i] +
\diSe[\vp_e] \Big) d\sigma,
\end{equation*}
\begin{align*}
\int_{\Om \setminus \overline{D}} |\nabla u|^2 dx &= -\int_{\p
D} \bar{u} \Big( \frac{1}{2}\vp_i + \Kcal_{\Gamma_i}^*
[\vp_i] + \diSe[\vp_e] \Big) d\sigma \\
& \qquad +  \int_{\p \Om}
\bar{u} \Big( -\frac{1}{2}\vp_e + \Kcal_{\Gamma_e}^* [\vp_e]
+ \deSi[\vp_i] \Big) d\sigma,
\end{align*}
and
\begin{equation*}
\int_{\RR^2 \setminus \overline{\Om}} |\nabla u|^2 dx = -\int_{\p
\Om} \bar{u} \Big( \frac{1}{2}\vp_e + \Kcal_{\Gamma_e}^* [\vp_e] +
\deSi[\vp_i] \Big)  d\sigma .
\end{equation*}
Summing up the above three identities we find
\begin{equation*}
\begin{split}
\int_{\RR^2} |\nabla u|^2 dx &  = - \int_{\p D} \bar{u} \vp_i d\sigma - \int_{\p \Om} \bar{u} \vp_e d\sigma \\
& = \langle \Phi, -\SSS [\Phi] \rangle_\Hcal .
\end{split}
\end{equation*}
Thus $-\SSS \ge 0$. This completes the proof.
\end{proof}

To prove that $\KK^*$ is symmetrizable, we shall make use of the
following lemma which can be proved by Green's formulas.

\begin{lem} \label{prop S and D}
Let $E \subset \RR^2$ be a bounded domain.
\begin{itemize}
\item[{\rm (i)}] If $u$ is a solution of $\Delta u =0$ in $E$, then
\begin{equation}\label{greeninner}
\Scal_{\p E} \Big[\frac{\p u}{\p \nu} \Big{|}_- \Big] (x) = \Dcal_{\p E} \Big[u\big{|}_- \Big] (x),\quad x\in \RR^2 \setminus
\overline{E}.
\end{equation}

\item[{\rm (ii)}] If $u$ is a solution of
\begin{equation}
\begin{cases}
\Delta u = 0 \quad & \mbox{in } \RR^2 \setminus \overline{E}, \\
u(x) \to 0,  & |x| \to \infty,
\end{cases}
\end{equation}
then
\begin{equation*}
\Scal_{\p E} \Big[\frac{\p u}{\p \nu} \big{|}_+ \Big] (x) = \Dcal_{\p E} \Big[u\big{|}_+ \Big] (x),\quad x\in E.
\end{equation*}
\end{itemize}
\end{lem}

Note that the well-known Calder\'on's identity (also known as
Plemelj's symmetrization principle)
\begin{equation} \label{SK* KS}
\Scal_{\p E} \Kcal_{\p E}^* = \Kcal_{\p E} \Scal_{\p E}
\end{equation}
is an immediate consequence of Lemma \ref{prop S and D}. In fact, if we put $u=\Scal_{\p E}[\vp]$ in \eqnref{greeninner}, we have
\begin{equation*}
- \frac{1}{2} \Scal_{\p E}[ \vp ] (x) +  \Scal_{\p E} \Kcal_{\p
E}^*[\vp](x) = \Dcal_{\p E}  \Scal_{\p E} [\vp] (x), \quad x \in
\RR^2 \setminus \overline{E}.
\end{equation*}
By taking the limit as $x \to \p E$ from outside $E$, we obtain \eqref{SK*
KS} using the jump relation \eqnref{doublejump} of the double layer potential.

The following lemma is a generalization of Calder\'on's identity.
\begin{lem} \label{lemma symmetrizable}
Let $\SSS$ and $\KK$ be given by \eqref{eq S} and \eqref{eq:K*},
respectively. Then
\begin{equation}\label{calderon2}
\SSS\KK^* = \KK \SSS,
\end{equation}
{\it i.e.},  $\SSS\KK^*$ is self-adjoint.
\end{lem}

\begin{proof}
Notice that
\begin{equation*}
\SSS \KK^* = \begin{bmatrix}
-\Scal_{\Gamma_i} \Kcal_{\Gamma_i}^* + \Scal_{\Gamma_e} \deSi & &  -\Scal_{\Gamma_i} \diSe + \Scal_{\Gamma_e} \Kcal_{\Gamma_e}^* \\ & & \\
-\Scal_{\Gamma_i} \Kcal_{\Gamma_i}^* + \Scal_{\Gamma_e} \deSi & &
-\Scal_{\Gamma_i} \diSe + \Scal_{\Gamma_e} \Kcal_{\Gamma_e}^*
\end{bmatrix}
\end{equation*}
and
\begin{equation*}
\KK \SSS = \begin{bmatrix}
-\Kcal_{\Gamma_i} \Scal_{\Gamma_i} + \Dcal_{\Gamma_e} \Scal_{\Gamma_i} & & -\Kcal_{\Gamma_i} \Scal_{\Gamma_e} + \Dcal_{\Gamma_e} \Scal_{\Gamma_e} \\ & & \\
-\Dcal_{\Gamma_i} \Scal_{\Gamma_i} +
\Kcal_{\Gamma_e}\Scal_{\Gamma_i} & & -\Dcal_{\Gamma_i}
\Scal_{\Gamma_e} + \Kcal_{\Gamma_e}\Scal_{\Gamma_e}
\end{bmatrix} .
\end{equation*}
We now check the following.
\begin{itemize}
\item $(\SSS\KK^*)_{11} = (\KK \SSS)_{11}$: by \eqnref{SK* KS} it follows that
$\Scal_{\Gamma_i} \Kcal_{\Gamma_i}^* = \Kcal_{\Gamma_i}
\Scal_{\Gamma_i}$ on $\Gamma_i$. If we set $u(x)=\Scal_{\Gamma_i}[\vp_i](x)$ and $E=\Omega$ in Lemma \ref{prop S and D} (ii), we have
    \begin{equation*}
    \Scal_{\Gamma_e} \deSi [\vp_i] = \Dcal_{\Gamma_e} \Scal_{\Gamma_i} [\vp_i] \quad\mbox{on } \Gamma_i .
    \end{equation*}
This implies $(\SSS\KK^*)_{11} = (\KK \SSS)_{11}$.

\item $(\SSS\KK^*)_{12} = (\KK \SSS)_{12}$: from Lemma \ref{prop S
and D} (ii), by setting $u(x)=\Scal_{\Gamma_e}[\vp_e](x)$ and
$E=D$ we find
    \begin{equation*}
    \Scal_{\Gamma_i} \diSe [\vp_e] (x)= \Dcal_{\Gamma_i} \Scal_{\Gamma_e} [\vp_e](x), \quad x \in \RR^2\setminus \overline{D}.
    \end{equation*}
    By taking the limit as $x \to \Gamma_i|_+$, we find
    \begin{equation} \label{SK*=KS 12 I}
    \Scal_{\Gamma_i} \diSe [\vp_e] = -\frac{1}{2} \Scal_{\Gamma_e} [\vp_e] +  \Kcal_{\Gamma_i} \Scal_{\Gamma_e} [\vp_e] \quad \mbox{on } \Gamma_i.
    \end{equation}
    Now, we use Lemma \ref{prop S and D} (ii) by taking $u = \Scal_{\Gamma_e}[\vp_e]$ and $E= \Omega$ and find
    \begin{equation*}
    \Scal_{\Gamma_e} \Big[\frac{\p \Scal_{\Gamma_e} [\vp_e]}{\p \nu_e} \big{|}_+ \Big] (x) = \Dcal_{\Gamma_e} \Scal_{\Gamma_e} [\vp_e](x) \quad \textmd{for } x \in \Omega,
    \end{equation*}
    and thus we have
    \begin{equation} \label{SK*=KS 12 II}
    \frac{1}{2}\Scal_{\Gamma_e} [\vp_e] + \Scal_{\Gamma_e} \Kcal_{\Gamma_e}^* [\vp_e] = \Dcal_{\Gamma_e} \Scal_{\Gamma_e} [\vp_e] \quad \textmd{on } \Gamma_i.
    \end{equation}
    Summing up \eqref{SK*=KS 12 I} and \eqref{SK*=KS 12 II} we find that $(\SSS\KK^*)_{12} = (\KK \SSS)_{12}$.

\item $(\SSS\KK^*)_{21} = (\KK \SSS)_{21}$: we use Lemma \ref{prop
S and D} (i) by setting $u=\Scal_{\Gamma_i}[\vp_i]$ and $E=D$ and
find
    \begin{equation*}
    \Scal_{\Gamma_i} \Big[\frac{\p \Scal_{\Gamma_i} [\vp_i]}{\p \nu_i} \big{|}_- \Big] (x) = \Dcal_{\Gamma_i} \Scal_{\Gamma_i} [\vp_i](x) \quad \textmd{for } x \in \RR^2\setminus \overline{D},
    \end{equation*}
    and thus we have
    \begin{equation} \label{SK*=KS 21 I}
    -\frac{1}{2}\Scal_{\Gamma_i} [\vp_i] + \Scal_{\Gamma_i} \Kcal_{\Gamma_i}^* [\vp_i] = \Dcal_{\Gamma_i} \Scal_{\Gamma_i} [\vp_i] \quad \textmd{on } \Gamma_e.
    \end{equation}
    By setting $u = \Scal_{\Gamma_i}[\vp_i]$ and $E=\Omega$ in Lemma \ref{prop S and D} (ii) we find
    \begin{equation*}
    \Scal_{\Gamma_e} \deSi [\vp_i] (x) = \Dcal_{\Gamma_e} \Scal_{\Gamma_i} [\vp_i](x),\quad x\in \Omega,
    \end{equation*}
    and by taking the limit as $x \to \Gamma_e|_-$, we find
    \begin{equation} \label{SK*=KS 21 II}
    \Scal_{\Gamma_e} \deSi [\vp_i] = \frac{1}{2} \Scal_{\Gamma_i} [\vp_i] +  \Kcal_{\Gamma_e} \Scal_{\Gamma_i} [\vp_i],\quad \mbox{on } \Gamma_e.
    \end{equation}
    Summing up \eqref{SK*=KS 21 I} and \eqref{SK*=KS 21 II} we find that $(\SSS\KK^*)_{21} = (\KK \SSS)_{21}$.

\item $(\SSS\KK^*)_{22} = (\KK \SSS)_{22}$: by \eqnref{SK* KS} it follows that
$\Scal_{\Gamma_e} \Kcal_{\Gamma_e}^* = \Kcal_{\Gamma_e}
\Scal_{\Gamma_e}$ on $\Gamma_e$. Thus, we have only to prove that
    \begin{equation*}
    \Scal_{\Gamma_i} \diSe [\vp_e] = \Dcal_{\Gamma_i} \Scal_{\Gamma_e} [\vp_e] \quad\mbox{on } \Gamma_e,
    \end{equation*}
    which follows from Lemma \ref{prop S and D} (i) by setting $u(x)=\Scal_{\Gamma_e}[\vp_e](x)$ and $E=D$.
\end{itemize}
This completes the proof.
\end{proof}

\subsection{$\KK^*$ is symmetrizable}

Let $\mathcal{C}_p(\Hcal)$, $1 \leq p < \infty$, be the
Schatten-von Neumann class of compact operators acting on $\Hcal$
(see \cite{GK}). We recall that a compact operator $A$ on $\Hcal$
is in the Schatten-von Neumann class $\mathcal{C}_p(\Hcal)$, with
$1 \leq p < \infty$, if the sequence of its singular values is in
$l_p=\{(\mu_n)_{n\in \ZZ} : \sum_{n \in \ZZ} |\mu_n|^p <
\infty\}$. An equivalent characterization is $\sum_n ||
A\Phi_n||^p < \infty$ for any orthonormal basis $(\Phi_n)$ of
$\Hcal$. The elements of $\mathcal{C}_2(\Hcal)$ are the
Hilbert-Schmidt operators. It is proved in \cite{KPS} that
$\Kcal_{\Gamma_i}^* \in \mathcal{C}_2 (L^2(\Gamma_i))$ and
$\Kcal_{\Gamma_e}^* \in \mathcal{C}_2 (L^2(\Gamma_e))$ are
Hilbert-Schmidt operators. On the other hand, $\pd{}{\nu_i}
\Scal_{\Gamma_e}$ and $\pd{}{\nu_e} \Scal_{\Gamma_i}$ are
Hilbert-Schmidt operators on $L^2(\Gamma_i)$ and $L^2(\Gamma_e)$,
respectively, because they have smooth integral kernels. Thus they
belong to $\mathcal{C}_2$. So we easily have the following lemma.

\begin{lem} \label{lemma KK* Schatten}
$\KK^* \in \mathcal{C}_2(\Hcal)$.
\end{lem}

By Lemma \ref{lemma S positive}, $-\SSS$ is self-adjoint and $-\SSS \ge 0$ on $\Hcal$. Thus
there exists a unique square root of $-\SSS$ which we denote by
$\sqrt{-\SSS}$; furthermore, $\sqrt{-\SSS}$ is self-adjoint and $\sqrt{-\SSS} \ge 0$ (see for instance Theorem 13.31 in \cite{Ru}). We now look into the kernel of $\SSS$. If $\Phi=(\vp_i, \vp_e) \in \mbox{Ker} (\SSS)$, then the function $u$ defined by
 $$
 u(x):= \Scal_{\Gamma_i} [\vp_i](x) + \Scal_{\Gamma_e} [\vp_e](x), \quad x \in \RR^2
 $$
satisfies $u=0$ on  $\Gamma_i$ and $\Gamma_e$. Therefore, $u(x)=0$ for all $x \in \Om$.
It then follows from \eqnref{singlejump} that $\vp_i=0$ and
 \beq\label{halfeigen}
 \Kcal_{\Gamma_e}^*[\vp_e]= \frac{1}{2} \vp_e \quad\mbox{on } \Gamma_e.
 \eeq
If $\vp_e \in L^2_0(\Gamma_e)$, then $u(x) \to 0$ as $|x| \to
\infty$ , and hence $u(x)=0$ for $x \in \RR^2 \setminus \Om$ as
well. Thus $\vp_e =0$. The eigenfunctions of \eqnref{halfeigen}
make a one dimensional subspace of $L^2(\Gamma_e)$, which means that
$\mbox{Ker}(\SSS)$ is of at most one dimension.

We now recall a result of Khavinson {\it et al} \cite[proof of Theorem 1]{KPS}: let $M \in \mathcal{C}_p(\Hcal)$. If there exists a strictly positive bounded self-adjoint operator $R$ such that $R^2 M$ is self adjoint, then there is a bounded self-adjoint operator $A \in \mathcal{C}_p(\Hcal)$ such that
 \beq
 AR=RM.
 \eeq
We use this result and \eqnref{calderon2} to show that there is a bounded self-adjoint operator $\AAA$ on $\mbox{Ran}(\SSS)$ such that
 \beq\label{ASSK}
 \AAA \sqrt{-\SSS} = \sqrt{-\SSS} \KK^*.
 \eeq
By defining $\AAA$ to be $0$ on $\mbox{Ker}(\SSS)$, we extend $\AAA$ to $\Hcal$. We note that \eqnref{ASSK} still holds and the extended operator is self-adjoint in $\Hcal$. In fact, if $\Phi \in \mbox{Ker}(\SSS)$, then $\KK^*[\Phi]= \frac{1}{2} \Phi$ because of \eqnref{halfeigen}, and hence $\sqrt{-\SSS} \KK^*[\Phi]=0$. Moreover, if $\Phi, \Psi \in \Hcal$, then we can decompose them as $\Phi=\Phi_1 + \Phi_2$ and $\Psi=\Psi_1 + \Psi_2$ where $\Phi_1, \Psi_1 \in \mbox{Ran}(\SSS)$ and $\Phi_2, \Psi_2 \in \mbox{Ker}(\SSS)$. Let $\Phi_1= \sqrt{-\SSS} \tilde\Phi_1$ and $\Psi_1= \sqrt{-\SSS} \tilde\Psi_1$. We then get
 \begin{align*}
 & \langle \AAA \Phi, \Psi \rangle = \langle \AAA \Phi_1, \Psi \rangle = \langle \AAA \sqrt{-\SSS} \tilde \Phi_1, \Psi \rangle = \langle \sqrt{-\SSS} \KK^* \tilde \Phi_1, \Psi \rangle \\
 & = \langle \sqrt{-\SSS} \KK^* \tilde \Phi_1, \Psi_1 \rangle = \langle \AAA \Phi_1, \Psi_1 \rangle = \langle \Phi_1, \AAA \Psi_1 \rangle = \langle \Phi, \AAA \Psi \rangle,
 \end{align*}
and hence $\AAA$ is self-adjoint on $\Hcal$.

We obtain the following theorem.
\begin{thm} \label{thm A}
There exists a bounded self-adjoint operator $\AAA \in
\mathcal{C}_2(\Hcal)$ such that
\begin{equation}\label{A sqrtS}
\AAA \sqrt{-\SSS} = \sqrt{-\SSS} \KK^*.
\end{equation}
\end{thm}

\section{Limiting properties of the solution and the electromagnetic power dissipation}

Let $V_\delta$ be the solution to \eqnref{basiceqn} with $\alpha=1$. In this
section we derive a necessary and sufficient condition on the
source $f$, which is supported outside $\overline{\Om}$, such that
the blow-up \eqnref{blowup1} of the power dissipation takes place.

The solution $V_\delta$ can be represented as
 \beq\label{vdelta2}
 V_\delta(x) = F(x) + \Scal_{\Gamma_i}[\vp_i^\delta](x) + \Scal_{\Gamma_e}[\vp_e^\delta](x),
 \eeq
where $\Phi_\delta = (\vp_i^\delta, \vp_e^\delta) \in
L^2_0(\Gamma_i) \times L^2_0(\Gamma_e)$ is the solution to
\eqnref{matrixeq4 ep=1}. Since $\int_{\Om\setminus \overline{D}}
|\nabla F|^2 dx < \infty$, \eqnref{blowup1} occurs if and only if
\beq \delta \int_{\Om\setminus \overline{D}} \left| \nabla
(\Scal_{\Gamma_i}[\vp_i^\delta] + \Scal_{\Gamma_e}[\vp_e^\delta])
\right|^2 dx \to \infty \quad\mbox{as } \delta \to \infty. \eeq
One can use \eqnref{singlejump} to obtain
 $$
 \int_{\Om\setminus \overline{D}} \left| \nabla (\Scal_{\Gamma_i}[\vp_i^\delta] + \Scal_{\Gamma_e}[\vp_e^\delta]) \right|^2 dx
 = -\frac{1}{2} \langle \Phi_\delta, \SSS \Phi_\delta \rangle + \langle \KK^* \Phi_\delta, \SSS \Phi_\delta \rangle,
 $$
where $\langle ~, ~ \rangle$ is the Hermitian product on $\Hcal$. We then get from \eqnref{A sqrtS}
 \beq
 \int_{\Om\setminus \overline{D}} \left| \nabla (\Scal_{\Gamma_i}[\vp_i^\delta] + \Scal_{\Gamma_e}[\vp_e^\delta]) \right|^2 dx
 = \frac{1}{2} \langle \sqrt{-\SSS} \Phi_\delta, \sqrt{-\SSS} \Phi_\delta \rangle - \langle \AAA \sqrt{-\SSS} \Phi_\delta, \sqrt{-\SSS} \Phi_\delta \rangle.
 \eeq

Since $\AAA$ is
self-adjoint, we have an orthogonal decomposition
 \beq
 \Hcal = \Ker{\AAA} \oplus (\Ker{\AAA})^\perp,
 \eeq
and $(\Ker{\AAA})^\perp = \overline{\Range{\AAA}}$. Let $P$ and
$Q=I-P$ be the orthogonal projections from $\Hcal$ onto
$\Ker{\AAA}$ and $(\Ker{\AAA})^\perp$, respectively. Let $\lam_1,\lam_2,\ldots$ with $|\lam_1|
\ge |\lam_2| \ge \ldots$ be the nonzero eigenvalues of $\AAA$ and
$\Psi_n$ be the corresponding (normalized) eigenfunctions. Since
$\AAA \in \mathcal{C}_2(\Hcal)$, we have
 \beq
 \sum_{n=1}^\infty \lam_n^2 < \infty,
 \eeq
and
 \beq
 \AAA \Phi = \sum _{n=1}^\infty \lam_n \langle \Phi, \Psi_n \rangle \Psi_n , \quad \Phi \in \Hcal
 \eeq

We apply $\sqrt{-\SSS}$ to \eqref{matrixeq4 ep=1} to obtain
\begin{equation*}
(z_\de \sqrt{-\SSS} + \sqrt{-\SSS} \KK^*) \Phi_\de = \sqrt{-\SSS} g.
\end{equation*}
Then \eqnref{A sqrtS} yields
\begin{equation} \label{eq prof thm conv 1}
(z_\de \II + \AAA) \sqrt{-\SSS} \Phi_\de = \sqrt{-\SSS} g,
\end{equation}
and hence
 \begin{align*}
 P\sqrt{-\SSS} \Phi_\de &= \frac{1}{z_\delta } P\sqrt{-\SSS} g , \\
 z_\de Q \sqrt{-\SSS} \Phi_\de + \AAA Q \sqrt{-\SSS} \Phi_\de &= Q \sqrt{-\SSS} g  .
 \end{align*}
Thus we get
 $$
 Q \sqrt{-\SSS} \Phi_\de = \sum_n \frac{\langle Q\sqrt{-\SSS} g, \Psi_n
 \rangle}{\lam_n + z_\de} \Psi_n.
 $$
We also get
 $$
 \AAA \sqrt{-\SSS} \Phi_\de = \sum_n \frac{\lam_n \langle Q\sqrt{-\SSS} g, \Psi_n
 \rangle}{\lam_n + z_\de} \Psi_n.
 $$
Thus we have
 \beq
 \langle \sqrt{-\SSS} \Phi_\delta, \sqrt{-\SSS} \Phi_\delta \rangle = \frac{1}{|z_\delta|^2 } \| P\sqrt{-\SSS} g \|^2 + \sum_n \frac{|\langle Q\sqrt{-\SSS} g, \Psi_n \rangle|^2}{|\lam_n + z_\de|^2},
 \eeq
and
 \beq
 \langle \AAA \sqrt{-\SSS} \Phi_\delta, \sqrt{-\SSS} \Phi_\delta \rangle
 = \sum_n \frac{\lam_n |\langle Q\sqrt{-\SSS} g, \Psi_n \rangle|^2}{|\lam_n + z_\de|^2} .
 \eeq
Since $$|\lambda_n+z_\delta|^2 = \left(\lambda_n - \frac{\delta^2}{2(4+\delta^2)}\right)^2 + \frac{\delta^2}{(4+\delta^2)^2} \approx \lambda_n^2+\delta^2$$ and $\lam_n \to 0$ as $n \to \infty$, we have
 \beq\label{estScal}
 \int_{\Om\setminus \overline{D}} \left| \nabla (\Scal_{\Gamma_i}[\vp_i^\delta] + \Scal_{\Gamma_e}[\vp_e^\delta]) \right|^2 dx
 \approx \frac{1}{\delta^2} \| P\sqrt{-\SSS} g \|^2 + \sum_n \frac{|\langle Q\sqrt{-\SSS} g, \Psi_n \rangle|^2}{|\lam_n|^2 + \de^2}.
 \eeq
Here and throughout this paper $A \approx B$ means that there are constants
$C_1$ and $C_2$ such that
$$
C_1 A \le B \le C_2 A.
$$

We note that if $\mbox{Ker}(\KK^*)=\{ 0 \}$, then $P\sqrt{-\SSS}=0$. To see this let $\Phi_0$ be a basis of $\mbox{Ker}(\SSS)$. Then we have $\KK^* \Phi_0= \frac{1}{2} \Phi_0$.
If $\AAA \sqrt{-\SSS} \Phi =0$, then $\sqrt{-\SSS} \KK^* \Phi =0$ by \eqnref{A sqrtS}. Therefore $\KK^* \Phi \in \mbox{Ker}(\SSS)$. If $\mbox{Ker}(\KK^*)=\{ 0 \}$, then $\Phi= c\Phi_0$ for some constant $c$. This means that $P\sqrt{-\SSS}=0$.

We obtain the following theorem:
\begin{thm}\label{ALRgen}
If $P\sqrt{-\SSS} g \neq 0$, then \eqnref{blowup1} takes place. If $\mbox{Ker}(\KK^*)=\{ 0 \}$, then \eqnref{blowup1}
 takes place if and only if
\beq\label{gcloak} \de \sum_n \frac{|\langle \sqrt{-\SSS} g,
\Psi_n \rangle|^2}{\lam_n^2 + \de^2} \to \infty \quad \mbox{ as }
\de \to 0. \eeq
\end{thm}

The condition \eqnref{gcloak} gives a necessary and sufficient
condition on the source term $f$ for the blow up of the
electromagnetic power dissipation in $\Omega\setminus\overline{D}$ when $\alpha=1$. This
condition is in terms of the Newtonian potential of $f$.
In the next section, we explicitly compute the eigenvalues and eigenfunctions of $\AAA$ for the case
of an annulus configuration. In particular, we show the existence
of a  cloaking region such that if $f$ is supported outside that region,
then there is no blow up while if it is supported inside and satisfies
certain conditions, there is a blow up and CALR occurs.


\section{Anomalous resonance in an annulus}

In this section we consider the anomalous resonance when the domains
$\Om$ and $D$ are concentric disks. We calculate the explicit form
of the limiting solution. Throughout this section, we set $\Omega=
B_e= \{ |x| < r_e\}$ and $D=B_i= \{ |x| < r_i\}$, where $r_e>r_i$.

Let $\Gamma=\{ |x|=r_0\}$. One can easily see that for each integer $n$
\begin{equation} \label{Single circular}
\Scal_\Gamma[e^{in\theta}](x) = \begin{cases}
 \ds - \frac{r_0}{2|n|} \left(\frac{r}{r_0}\right)^{|n|}
 e^{in\theta} \quad & \mbox{if } |x|=r < r_0, \\
 \nm
 \ds - \frac{r_0}{2|n|} \left(\frac{r_0}{r}\right)^{|n|}
 e^{in\theta} \quad & \mbox{if } |x|=r > r_0,
 \end{cases}
\end{equation}
and hence
 \beq\label{single-cir-nor}
\frac{\p }{\p r} \Scal_\Gamma[e^{in\theta}](x) =  \begin{cases}
 \ds - \frac{1}{2} \left(\frac{r}{r_0}\right)^{|n|-1}
 e^{in\theta} \quad & \mbox{if } |x|=r < r_0, \\
 \nm
 \ds \frac{1}{2} \left(\frac{r_0}{r}\right)^{|n|+1}
 e^{in\theta} \quad & \mbox{if } |x|=r > r_0.
 \end{cases}
 \eeq
It then follows from \eqnref{singlejump} that
 \beq\label{kcal-circ}
 \Kcal_\Gamma^* [e^{in \theta}] =0 \quad \forall n \neq 0.
 \eeq
It is worth mentioning that this property was observed in
\cite{KS96} and immediately follows from the fact that
$$\Kcal_\Gamma^* [\varphi] = \frac{1}{4\pi r_0} \int_\Gamma
\varphi d\sigma .
$$
We also get from \eqref{adjoint diSe} and
\eqref{adjoint deSi}
\begin{equation*}
\Dcal_\Gamma[e^{in\theta}](x) = \begin{cases}
 \ds \frac{1}{2} \left(\frac{r}{r_0}\right)^{|n|}
 e^{in\theta} \quad & \mbox{if } |x|=r < r_0, \\
 \nm
 \ds -\frac{1}{2} \left(\frac{r_0}{r}\right)^{|n|}
 e^{in\theta} \quad & \mbox{if } |x|=r > r_0.
 \end{cases}
\end{equation*}

Because of \eqnref{kcal-circ} it follows that
\begin{equation*}
\KK^* = \begin{bmatrix} 0 & - \diSe \\ \deSi & 0 \end{bmatrix},
\end{equation*}
and hence we have from \eqnref{single-cir-nor} that
\begin{equation} \label{K* circ 1}
\KK^* \begin{bmatrix} e^{in\theta} \\ 0 \end{bmatrix} =
\frac{1}{2} \rho^{|n|+1} \begin{bmatrix} 0 \\ e^{in\theta}
\end{bmatrix}
\end{equation}
and
\begin{equation} \label{K* circ 2}
\KK^* \begin{bmatrix} 0 \\ e^{in\theta}  \end{bmatrix} =
\frac{1}{2} \rho^{|n|-1} \begin{bmatrix} e^{in\theta} \\ 0
\end{bmatrix}
\end{equation}
for all $n \neq 0$, where
\begin{equation*}
\rho = \frac{r_i}{r_e}.
\end{equation*}
Thus $\KK^*$ as an operator on $\Hcal$ has the trivial kernel, {\it i.e.},
\begin{equation}\label{kerzero}
\Ker \, \KK^* = \{0\}.
\end{equation}

According to \eqnref{K* circ 1} and \eqnref{K* circ 2}, if $\Phi$ is given by
 $$
 \Phi= \sum_{n \neq 0} \begin{bmatrix} \vp_i^n\\ \vp_e^n \end{bmatrix} e^{in\theta},
 $$
then
\begin{equation*}
\KK^* \Phi =   \sum_{n \neq 0}
\begin{bmatrix} \ds \frac{\rho^{|n|-1}}{2} \vp_e^n  \\
\nm
 \ds \frac{\rho^{|n|+1}}{2} \vp_i^n  \end{bmatrix} e^{in\theta} .
\end{equation*}
Thus, if $g$ is given by
 $$
 g= \sum_{n \neq 0} \begin{bmatrix} g_i^n\\ g_e^n \end{bmatrix} e^{in\theta},
 $$
the integral equations \eqnref{matrixeq4 ep=1} are equivalent to
\beq\label{inteqn2}
\begin{cases}
\ds z_\delta \vp_i^n + \frac{\rho^{|n|-1}}{2} \vp_e^n =  g_i^n , \\
\nm
\ds z_\delta \vp_e^n + \frac{\rho^{|n|+1}}{2} \vp_i^n =  g_e^n,
\end{cases}
\eeq
for every $|n|\geq 1$. It is readily seen that the solution
$\Phi=(\vp_i,\vp_e)$ to \eqnref{inteqn2} is given by
\begin{align*}
\vp_i &= 2 \sum_{n\neq 0} \frac{2z_\delta g_i^n-\rho^{|n|-1}g_e^n} {4z_\delta^2-\rho^{2|n|}}  e^{i n
\theta}, \\
\vp_e &= 2 \sum_{n\neq 0} \frac{2z_\delta g_e^n-\rho^{|n|+1}g_i^n} {4z_\delta^2-\rho^{2|n|}}  e^{i n
\theta}.
\end{align*}

If the source is located outside the structure, {\it
i.e.}, $f$ is supported in $\RR^2 \setminus \overline{B}_e$, then
the Newtonian potential of $f$, $F$, is harmonic in $B_{r_e}$ and
\beq\label{Ffourier} F(x)=c-\sum_{n\neq 0}
\frac{g_e^n}{|n|r_e^{|n|-1}}r^{|n|}e^{in\theta},
\eeq 
for $|x| \le r_e$, where $g$ is defined by \eqnref{defg}. Thus we have
 \beq
 g_i^n = - g_e^n \rho^{|n|-1}.
 \eeq
Here, $g_e^n$ is the Fourier coefficient of $-\frac{\p F}{\p
\nu_e}$ on $\Gamma_e$, or in other words,
 \beq
 -\frac{\p F}{\p \nu_e} = \sum_{n \neq 0} g_e^n e^{i n \theta}.
 \eeq
We then get
\begin{equation}\label{formulaofphi}
\begin{cases}
\ds\vp_i = -2 \sum_{n\neq 0} \frac{(2z_\delta +1)\rho^{|n|-1}g_e^n} {4z_\delta^2-\rho^{2|n|}}  e^{i n
\theta},\\
\ds\vp_e = 2 \sum_{n\neq 0} \frac{(2z_\delta +\rho^{2|n|})g_e^n} {4z_\delta^2-\rho^{2|n|}}  e^{i n
\theta}.
\end{cases}
\end{equation}
Therefore, from \eqref{Single circular} we find that
\beq\label{SSoutside} \Scal_{\Gamma_i}[\vp_i](x) +
\Scal_{\Gamma_e}[\vp_e](x) = \sum_{ n\neq 0} \frac{2(r_i^{2|n|} -
r_e^{2|n|})z_\delta}{|n|r_e^{|n|-1}(4z_\delta^2 - \rho^{2|n|})}
\frac{g_e^n}{r^{|n|}} e^{i n \theta} , \quad r_e < r=|x|, \eeq and
\begin{align}
\Scal_{\Gamma_i}[\vp_i](x)
& =  -\sum_{ n\neq 0} \frac{ r_i^{2|n|}(2z_\delta+1)}{|n|r_e^{|n|-1}(\rho^{2|n|}-4z_\delta^2)}
\frac{g_e^n}{r^{|n|}} e^{i n \theta} ,  \quad r_i < r=|x| < r_e, \label{singleexp1} \\
\Scal_{\Gamma_e}[\vp_e](x) & = \sum_{ n\neq 0}
\frac{(2z_\delta+\rho^{2|n|})}{|n|r_e^{|n|-1}(\rho^{2|n|}-4z_\delta^2)}
g_e^n r^{|n|} e^{i n \theta}, \quad r_i < r=|x| < r_e.
\label{singleexp2}
\end{align}

We next obtain the following lemma which provides essential estimates
for the investigation of this section.

\begin{lem} \label{lemest}
There exists $\delta_0$ such that
 \beq\label{essest}
 E_\delta := \int_{B_e\setminus \overline{B_i}} \delta |\nabla V_\delta|^2
 \approx \sum_{n \ne 0}
 \frac{\delta|g_e^n|^2}{|n|(\delta^2+\rho^{2|n|})}
 \eeq
uniformly in $\delta \le \delta_0$.
\end{lem}

\proof \, Using \eqnref{Ffourier}, \eqnref{singleexp1}, and
\eqnref{singleexp2}, one can see that
$$
V_\delta (x) =  c+ r_e \sum_{n\neq 0} \left[
\frac{r_i^{2|n|}}{r^{|n|}} (2z_\delta+1) - 6z_\delta r^{|n|}
\right] \frac{g_e^n e^{i n \theta}}{|n|r_e^{|n|}(4z_\delta^2 -
\rho^{2|n|})} .
$$
Then straightforward computations yield that
\begin{align*}
E_\delta & \approx r_e^2
\sum_{n\ne0} \delta (1+\rho^{2|n|}) \left|\frac{2z_\delta +1}
{4z_\delta^2-\rho^{2|n|}}\right|^2 (4|z_\delta|^2 - \rho^{2|n|})
\frac{|g_e^n|^2}{|n|} .
\end{align*}
If $\delta$ is sufficiently small, then one can also easily show
that
 $$
 |4z_\delta^2-\rho^{2|n|}| \approx \delta^2 + \rho^{2|n|}.
 $$
Therefore we get \eqnref{essest} and the proof is complete.
\qed

It is worth noticing that estimate \eqnref{essest} is exactly the
same as the one from Theorem \ref{ALRgen} since the eigenvalues
of $\AAA$ are $\{\pm \rho^{|n|}/2\}$. To see this fact, we
restrict the identity $\AAA \sqrt{-\SSS} = \sqrt{-\SSS}  \KK^*$
to the vectorial space spanned by $\begin{bmatrix} 0 \\
e^{in\theta}
\end{bmatrix}$ and $\begin{bmatrix} e^{in\theta} \\ 0
\end{bmatrix}$. Taking the trace and the determinant of the
restricted identity and using \eqnref{K* circ 1} and \eqnref{K*
circ 2} proves that the set of eigenvalues of $\AAA$ is $\{\pm
\rho^{|n|}/2\}$.

Now, we turn to Lemma \ref{lemest}. We investigate the behavior of the series in the right hand side of \eqnref{essest}.
Let
 \beq
 N_\delta = \frac{\log \delta}{\log \rho}.
 \eeq
If $|n| \le N_\delta$, then $\delta \le \rho^{|n|}$, and hence
 \beq \label{easier009}
\sum_{n \ne 0} \frac{\delta|g_e^n|^2}{|n|(\delta^2+\rho^{2|n|})} \ge  \sum_{0 \ne |n| \le N_\delta} \frac{\delta|g_e^n|^2}{|n|(\delta^2+\rho^{2|n|})} \ge \frac{1}{2} \sum_{0 \ne |n| \le N_\delta} \frac{\delta|g_e^n|^2}{|n|\rho^{2|n|}}.
 \eeq

Suppose that
\beq\label{easier}
\limsup_{|n| \rightarrow\infty} \frac{|g_e^n|^2}{|n|\rho^{|n|}}= \infty.
\eeq
Then there is a subsequence $\{n_k\}$ with $|n_1|<|n_2|<\cdots$ such that
\beq\label{subsequence}
\lim_{k \to \infty} \frac{|g_e^{n_k}|^2}{|n_k|\rho^{|n_k|}} = \infty.
\eeq
If we take $\delta=\rho^{|n_k|}$, then $N_\delta =|n_k|$ and
 \beq
 \sum_{0 \ne |n| \le N_\delta} \frac{\delta|g_e^n|^2}{|n|\rho^{2|n|}} = \rho^{|n_k|} \sum_{0 \ne |n| \le |n_k|} \frac{|g_e^n|^2}{|n|\rho^{2|n|}} \ge \frac{|g_e^{|n_k|}|^2}{|n_k|\rho^{|n_k|}}.
 \eeq
Thus we obtain from \eqnref{essest} that
 \beq
 \lim_{k \to \infty} E_{\rho^{|n_k|}} = \infty.
 \eeq

We emphasize that \eqnref{easier} is not enough to guarantee \eqnref{blowup1} as pointed out by Jianfeng Lu and Jens Jorgensen (private communication). In fact, if we let
\beq\label{genexam}
g_e^n = \begin{cases} n \rho^{n/2},& \mbox{if } n=2^j,~j=1,2,\ldots,\\
                      0, &\mbox{otherwise},
         \end{cases}
\eeq
and $\delta_k=\rho^{n_k}$ with $n_k=2^k + 2^{k-1}$ for $k=1,2,\ldots$, then
 $$
 \limsup_{n \rightarrow\infty} \frac{|g_e^n|^2}{|n|\rho^{|n|}}= \infty.
 $$
But one can easily see that $|2^j-n_k| \ge 2^{j-2}$ and
 $$
 \frac{\rho^{n_k +2^j}}{\rho^{2n_k}+\rho^{2^{j+1}}} < \rho^{|n_k -2^j|}, \quad j, k =1,2,\ldots.
 $$
Thus we obtain
\begin{align*}
\sum_{n \ne 0}
 \frac{\delta_{k} |g_e^n|^2}{|n|(\delta_{k}^2+\rho^{2|n|})} =\sum_{j=1}^\infty \frac{2^j \rho^{n_k +2^j}}{\rho^{2n_k}+\rho^{2^{j+1}}} \le \sum_{j=1}^\infty 2^j \rho^{|n_k -2^j|} \le \sum_{j=1}^\infty 2^j \rho^{2^{j-2}} <\infty,
\end{align*}
which means that
 \beq
 E_{\delta_k} \le C
 \eeq
regardless of $k$. It is worth mentioning that the $g_e^n$ defined by \eqnref{genexam} are certainly Fourier coefficients of $-\frac{\p F}{\p\nu_e}$ on $\Gamma_e$ for an $F$ which is harmonic in $B_{r_*}$, 
given by \eqnref{Ffourier} when  $|x| \le r_*$. Also there is a source function which generates these Fourier coefficients. To see this,
choose $r_1$ and $r_2$ with $r_e<r_1<r_2<r_*$ and let $\tau(r)$, be a function which is 1 for $r<r_1$, and zero for $r>r_2$ 
and which smoothly interpolates between these values in the interval $r_1\leq r \leq r_2$. Then we see
that $\widetilde F(x)$ defined to be zero for $|x|\geq r_2$ and equal to $\tau(|x|)F(x)$ for $|x|<r_2$, has the same Fourier coefficients $g_e^n$ as $F$ on $\Gamma_e$, and the associated source 
function $\widetilde f=\Delta \widetilde F$ is supported 
in the annulus between $|x|=r_1$ and $|x|=r_2$. However, it is not clear whether the Fourier coefficients can be realized as being associated with 
a Newtonian potential of a source function whose support is located outside the radius $r_e$ and not surrounding the annulus.

We now impose an additional condition.
We assume that $\{g_e^n\}$ satisfies the following gap property:
\begin{itemize}
\item[GP]: There exists  a sequence $ \{n_k\} $ with $|n_1|<|n_2|<\cdots$ such that
$$
\lim_{k \to \infty} \rho^{|n_{k+1}|-|n_k|}\frac{|g_e^{n_k}|^2}{|n_k|\rho^{|n_k|}} = \infty.
$$
\end{itemize}
If GP holds, then we immediately see that \eqnref{easier} holds,
but the converse is not true. If \eqnref{easier} holds, \ie, there is a subsequence $\{n_k\}$ with $|n_1|<|n_2|<\cdots$ satisfying \eqnref{subsequence} and the gap $|n_{k+1}|-|n_k|$ is bounded, then GP holds. In particular, if
 \beq\label{limit}
 \lim_{n\rightarrow\infty} \frac{|g_e^n|^2}{|n|\rho^{|n|}}= \infty,
 \eeq
then GP holds.

Assume that $\{g_e^n\}$ satisfies GP and $\{n_k\}$ is such a sequence. Let $\delta=\rho^{\alpha}$ for some $\alpha$ and let $k(\alpha)$ be the number such that
$$ |n_{k(\alpha)}| \le \alpha < |n_{k(\alpha)+1}|.$$
Then, we have
 \beq\label{66}
 \sum_{0 \ne |n| \le N_{\delta}} \frac{\delta|g_e^n|^2}{|n|\rho^{2|n|}} = \rho^{\alpha} \sum_{0 \ne |n| \le \alpha} \frac{|g_e^n|^2}{|n|\rho^{2|n|}}\ge \rho^{|n_{k(\alpha)+1}|-|n_{k(\alpha)}|}\frac{|g_e^{n_{k(\alpha)}}|^2}{|n_{k(\alpha)}|\rho^{|n_{k(\alpha)}|}} \rightarrow \infty,
 \eeq
 as $\alpha\to\infty$.

We obtain the following lemma:
\begin{lem} \label{mainthm}
If \eqnref{easier} holds, then
 \beq\label{semiblowup}
 \limsup_{\delta\to 0} E_\delta = \infty.
 \eeq
If $\{g_e^n\}$ satisfies the condition GP,
then
 \beq\label{mainblowup}
\lim_{\delta\to 0} E_\delta = \infty.
 \eeq
\end{lem}

Suppose that the source function is supported inside the radius $r_*= \sqrt{r_e^3r_i^{-1}}$. Then its Newtonian potential cannot be extended harmonically in $|x| < r_*$ in general.
So, if $F$ is given by
 \beq\label{fser}
 F = c- \sum_{n\ne 0} a_n r^{|n|} e^{in\theta},\quad r< r_e,
 \eeq
then the radius of convergence is less than $r_*$. Thus we have
 \beq\label{easier22}
 \limsup_{|n|\rightarrow\infty} |n| |a_n|^2 r_*^{2|n|}= \infty,
 \eeq
\ie, \eqnref{easier} holds. The GP condition is equivalent to that there exists  $ \{n_k\} $ with   $|n_1|<|n_2|<\cdots$ such that
\beq\label{easier2}
\lim_{k \to \infty} \rho^{|n_{k+1}|-|n_k|}|n_k| |a_{n_k}|^2 r_*^{2|n_k|} =\infty .
\eeq

The following is the main theorem of this section.

\begin{thm}\label{corollary1}
Let $f$ be a source function supported in $\RR^2 \setminus \overline{B}_e$ and $F$ be the Newtonian potential of $f$.
\begin{itemize}
\item[{\rm (i)}] If $F$ does not extend as a harmonic function in $B_{r_*}$, then weak CALR occurs, \ie,
 \beq\label{yesblowup}
 \limsup_{\delta \to 0} E_\delta = \infty
 \eeq
and \eqnref{boundedness} holds with $a={r_e^2}/{r_i}$.

\item[{\rm (ii)}] If the Fourier coefficients of $F$ satisfy \eqnref{easier2}, then CALR occurs, \ie,
 \beq\label{yesblowup2}
 \lim_{\delta \to 0} E_\delta = \infty
 \eeq
and \eqnref{boundedness} holds with $a={r_e^2}/{r_i}$.

\item[{\rm (iii)}] If $F$ extends as a harmonic function in a neighborhood of $\overline{B_{r_*}}$, then CALR does not occur, \ie,
 \beq\label{noblowup}
 E_\delta < C
 \eeq
for some $C$ independent of $\delta$.
\end{itemize}
\end{thm}
\proof
If $F$ does not extend as a harmonic function in $B_{r_*}$, then \eqnref{easier} holds. Thus we have \eqnref{yesblowup}. If \eqnref{easier2} holds, then \eqnref{yesblowup2} holds by Lemma \ref{mainthm}. Moreover, by \eqnref{SSoutside}, we see that
\begin{align*}
|V_\delta| &\le  |F|+ \sum_{ n\neq 0} \left|\frac{2(r_i^{2|n|} -
r_e^{2|n|})z_\delta}{|n|r_e^{|n|-1}(4z_\delta^2 - \rho^{2|n|})}
\frac{g_e^n}{r^{|n|}} \right| \le|F|+ C\sum_{ n\neq 0} \frac{\delta r_e^{|n|}}{(\delta^2+ \rho^{2|n|})|n|r^{|n|}}\\
&\le|F|+ C\sum_{ n\neq 0} \frac{r_e^{2|n|}}{|n|r_i^{|n|}r^{|n|}} < C
, \quad\mbox{if}\quad r=|x|> \frac{r_e^2}{r_i}
\end{align*}
for some constants $C$ which may differ at each occurrence.

If $F$ extends as a harmonic function in a neighborhood of $\overline{B_{r_*}}$, then the power series of $F$, which is given by \eqnref{Ffourier},
converges for $r < r_*+2\epsilon$ for some $\epsilon >0$. Therefore there exists a constant $C$ such that
 $$
 \frac{|g_e^n|}{|n|r_e^{|n|-1}} \le C \frac{1}{(r_*+\ep)^{|n|}}
 $$
for all $n$. It then follows that
 \beq\label{gnest}
 |g_e^n| \le C  (r_e^2 \rho^{-1}+ r_e \ep)^{-|n|/2}  \le (\rho^{-1}+\ep)^{-|n|/2}
 \eeq
for all $n$. This tells us that
$$\sum_{n \ne 0}
\frac{\delta|g_e^n|^2}{|n|(\delta^2+\rho^{2|n|})} \le \sum_{n \ne 0} \frac{|g_e^n|^2}{2|n|\rho^{|n|}} \le \sum_{n \ne 0} \frac{1}{2|n| (1+ \ep \rho)^{|n|} }. $$
This completes the proof.
\qed

\medskip

If $f$ is a dipole in $B_{r_*} \setminus \overline{B}_e$, \ie,
$f(x)=a\cdot\nabla\delta_y(x)$ for a vector $a$ and $y\in B_{r_*}\setminus
\overline{B}_e$ where $\delta_y$ is the Dirac delta function at $y$. Then $F(x)=a\cdot\nabla G(x-y)$. From the  expansion of the fundamental solution
\beq\label{Gammaexp}
 G(x-y)
 =\sum_{n=1}^{\infty}\frac{-1}{2\pi n}\left[\frac{\cos n\theta_y}{r_y^n}r^n\cos n\theta
 +\frac{\sin n\theta_y}{r_y^n}r^n\sin n\theta\right]+ C,
 \eeq
we see that the Fourier coefficients of $F$ has the growth rate $r_y^{-n}$ and satisfies \eqnref{easier2}, and hence
CALR takes place. Similarly CALR takes place for a sum of dipole souces at different fixed positions in  $B_{r_*} \setminus \overline{B}_e$.
We emphasize that this fact was found in \cite{MN_PRSA_06}.

If $f$ is a quadrapole, \ie, $f(x)=A:\nabla\nabla\delta_y(x) = \sum_{i,j=1}^2 a_{ij} \frac{\p^2}{\p x_i \p x_j} \delta_y(x)$ for a $2 \times 2$ matrix $A=(a_{ij})$ and $y\in B_{r_*}\setminus
\overline{B}_e$. Then $F(x)=\sum_{i,j=1}^2 a_{ij} \frac{\p^2 G(x-y)}{\p x_i \p x_j}$.
Thus CALR takes place. This is in agreement with the numerical result in \cite{osa}.

If $f$ is supported in $\RR^2 \setminus \overline{B}_{r_*}$, then $F$ is harmonic in a neighborhood of $\overline{B}_{r_*}$, and hence CALR does not occur by Theorem \ref{corollary1}. In fact, we can say more about the behavior of the solution $V_\delta$ as $\delta \to 0$ which is related to the observation in \cite{NMM_94,MNMP_PRSA_05} that in the limit $\delta\to 0$
the annulus itself becomes invisible to sources that are sufficiently far away.

\begin{thm} \label{thmnotcloaked2}
If $f$ is supported in $\RR^2 \setminus \overline{B}_{r_*}$, then \eqnref{noblowup} holds (with $\alpha=1$ in \eqnref{basiceqn}).
Moreover, we have
 \beq\label{VdeltaF}
 \sup_{|x| \ge r_*} |V_\delta(x) - F(x) | \rightarrow 0 \quad\mbox{as}\quad \delta \rightarrow 0 .
 \eeq
\end{thm}

\proof
Since $\mbox{supp}\, f \subset \RR^2 \setminus \overline{B}_{r_*}$, the power series of $F$, which is given by \eqnref{Ffourier}, converges for $r < r_*+2\epsilon$ for some $\epsilon >0$.

According to \eqnref{SSoutside}, if $ r_e < r=|x|,$ then we have
$$V_\delta(x) - F(x)= \sum_{ n\neq 0} \frac{(r_e^{2|n|} - r_i^{2|n|})z_\delta}{|n|r_e^{|n|-1}(\rho^{2|n|}-4z_\delta^2)} \frac{g_e^n}{r^{|n|}} e^{i n \theta}.
$$
If $|x|= r_*$, then the identity
\begin{align*}
\frac{(r_e^{2|n|} -
r_i^{2|n|})z_\delta}{|n|r_e^{|n|-1}(\rho^{2|n|}-4z_\delta^2)}
\frac{g_e^n}{r_*^{|n|}} = \frac{(1 -
\rho^{2|n|})z_\delta}{(\rho^{|n|}-4z_\delta^2\rho^{-|n|})}
\frac{g_e^n r_*^{|n|}}{|n|r_e^{|n|-1}}
\end{align*}
holds and
\begin{align*}
&\left|\frac{(1 - \rho^{2|n|})z_\delta}{(\rho^{|n|}-4z_\delta^2\rho^{-|n|})}\right| \le  \left|\frac{1}{(z_\delta^{-1}\rho^{|n|}- z_\delta\rho^{-|n|})}\right|\\
&\le \left|\frac{1}{\Im(z_\delta^{-1}\rho^{|n|}-z_\delta\rho^{-|n|})}\right|
= \left( \frac{\delta}{4+\delta^2} \rho^{-|n|} + \frac{1}{\delta}\rho^{|n|}\right)^{-1}.
\end{align*}
It then follows from \eqnref{gnest} that
 $$
 |V_\delta(x) - F(x)| \le 2 \sum_{ n\neq 0} \left( \frac{\delta}{4+\delta^2} \rho^{-|n|} + \frac{1}{\delta}\rho^{|n|}\right)^{-1} \frac{r_e}{|n|} \left( \frac{\rho^{-1}}{\rho^{-1} + \ep} \right)^{|n|/2},
 $$
and hence
 $$
 |V_\delta(x) - F(x)| \rightarrow 0 \quad\mbox{as } \delta \rightarrow 0.
 $$
Since $V_\delta -F$ is harmonic in $|x| > r_e$ and tends to $0$ as $|x| \to \infty$, we obtain \eqnref{VdeltaF} by the maximum principle. This completes the proof.
\qed

Theorem \ref{thmnotcloaked2} shows that any source supported outside ${B}_{r_*}$ cannot make the blow-up of the power dissipation
happen and is not cloaked. In fact, it is known that we can recover the source $f$ from its Newtonian potential $F$ outside $B_{r_*}$ since $f$ is supported outside $\overline{B}_{r_*}$ (see \cite{Isa90}). Therefore we infer from \eqnref{VdeltaF} that $f$ may be recovered approximately by observing $V_\delta$ outside $B_{r_*}$.

\section{Conclusion}

In this paper we have provided for the first time a mathematical
justification of cloaking due to anomalous localized resonance in the
case of general source terms.  In particular,
we obtained an explicit necessary and sufficient condition on the
source term in order for CALR to take place. In the case of an annulus structure we show that weak CALR takes place for almost any source supported inside the critical radius. We also find a sufficient condition on the Fourier coefficients of the Newtonian potential of the source function for CALR to occur. It would be quite interesting to clarify whether weak CALR implies CALR or not for sources whose support does not completely surround the annulus.

The results and techniques of this paper can be immediately
extended to the three-dimensional case. The compact operator
$\KK^*$ is in the Schatten Von-Neumann class
$\mathcal{C}_p(L^2(\Gamma_i) \times L^2(\Gamma_e))$ for some
$1\leq p < \infty$, provided that $\Omega$ and $D$ are of class
$\mathcal{C}^{1,\alpha}$ for $0< \alpha<1$, and consequently, it
is symmetrizable.

\section*{Acknowledgements}
We thank the group of Jens Jorgensen, Robert Kohn, Jianfeng Lu, and Michael Weinstein for pointing out an error in section 5 of an earlier version of this paper and helping to clarify the distinction between CALR and weak CALR.

\end{document}